\documentclass{compositio}
\usepackage{amsmath}
\usepackage{amssymb}
\usepackage{amscd}
\usepackage{hyperref}

\renewcommand{\contentsname}{Contents\\{\footnotesize\normalfont(A table
of contents should normally not be included)}}
\newtheorem{theorem}{Theorem}[section]
\newtheorem{thm}[theorem]{Theorem}
\newtheorem{cor}[theorem]{Corollary}
\newtheorem{lem}[theorem]{Lemma}
\newtheorem{prop}[theorem]{Proposition}

\theoremstyle{definition}
\newtheorem{defn}[theorem]{Definition}

\newtheorem{ques}[theorem]{Question}

\newtheorem{rem}[theorem]{Remark}

\newcommand{\OO}{\mathcal{O}}
\newcommand{\mcX}{\mathcal{X}}
\newcommand{\mbb}{\mathbb}
\newcommand{\QQ}{\mbb{Q}}

\newcommand{\ZZ}{\mbb{Z}}

\newcommand{\RR}{\mbb{R}}
\newcommand{\PP}{\mbb{P}}
\newcommand{\bH}{\mbb{H}}

\begin{document}

\title{One cycles on rationally connected varieties}
\author{Zhiyu Tian }
\email{tian@caltech.edu}
\address{
Department of Mathematics 253-37\\
California Institute of Technology\\
Pasadena, CA, 91125}

\author{Hong R. Zong}
\email{rzong@math.princeton.edu}
\address{
Department of Mathematics\\
Princeton University\\
Princeton, NJ, 08544-1000}


\begin{abstract}
We prove that every curve on a separably rationally connected variety is rationally equivalent to a (non-effective) integral sum of rational curves. That is, the Chow group of $1$-cycles is generated by rational curves. Applying the same technique, we also show that the Chow group of $1$-cycles on a separably rationally connected Fano complete intersections of index at least $2$ is generated by lines.

 As a consequence, we give a positive answer to a question of Professor Totaro about integral Hodge classes on rationally connected 3-folds. And by a result of Professor Voisin, the general case is a consequence of the Tate conjecture for surfaces over finite fields.
\end{abstract}

\maketitle

\section{Introduction}

In \cite{kollarportugaliae} and \cite{voisin}, the following question is asked by Professor J\'anos Koll\'ar and Professor Claire Voisin (and in the case of dimension 3 by Professor Burt Totaro):
\begin{ques}\label{main question}
Is every integral Hodge $(n-1,n-1)$-class on a smooth projective rationally connected $n$-dimensional variety over $\mathbb{C}$ a $\mathbb{Z}$--linear combination of cohomology classes of rational curves?
\end{ques}

This question can be separated into two questions, as in \cite{voisin}.
\begin{itemize}
\item  Is every integral Hodge $(n-1,n-1)$-class on a smooth projective rationally connected $n$-fold over $\mathbb{C}$ a $\mathbb{Z}$-linear combination of the cohomology classes of curves?
\item  Is every curve class on a smooth projective rationally connected $n$-fold over $\mathbb{C}$ a $\mathbb{Z}$-linear combination of the cohomology classes of rational curves?
\end{itemize}
The first question is known in some special cases.
\begin{theorem}\label{voisin}
Let $X$ be a smooth projective variety over $\mathbb{C}$. Then every integral Hodge $(n-1,n-1)$-class on $X$ is a $\mathbb{Z}$-linear combination of the cohomology classes of curves if $X$ is one of the followings.
 \begin{itemize}
 \item $X$ is a uniruled or Calabi-Yau $3$-fold.\cite{voisinuniruled}
 \item $X$ is a Fano $4$-fold. \cite{HoringVoisinFano}
 \item $X$ is a Fano $5$-fold with index $2$. \cite{HoringVoisinFano}
 \item $X$ is a Fano $n$-fold ($n\geq 8$) with index $n-3$. \cite{Floris}
\end{itemize}
\end{theorem}

It is shown in \cite{voisin} that a positive answer to the first question in general is implied by the Tate conjecture for divisor classes on surfaces defined over finite fields.

Our main theorem gives a positive answer to the second question.

\begin{thm}\label{thm:main}
Let $X$ be a smooth proper and separably rationally connected variety over an algebraically closed field. Then every $1$-cycle is rationally equivalent to a $\mathbb{Z}$-linear combination of the cycle classes of rational curves. That is, the Chow group $CH_1(X)$ is generated by rational curves.
\end{thm}

Recall that rational equivalence implies cohomological equivalence over the complex numbers $\mathbb{C}$. Thus combining Theorem \ref{voisin} and Theorem \ref{thm:main}, we have:
\begin{cor}
Let $X$ be a smooth projective variety over $\mathbb{C}$. Then every integral Hodge $(n-1,n-1)$-class on $X$ is a $\mathbb{Z}$-linear combination of the cohomology classes of rational curves if $X$ is one of the followings.
 \begin{itemize}
 \item X is a rationally connected 3-fold.
 \item X is a Fano $4$-fold.
 \item X is a Fano $5$-fold with index $2$.
 \item X is a Fano $n$-fold ($n\geq 8$) with index $n-3$.
\end{itemize}
\end{cor}
In particular, the $3$-fold case of Question \ref{main question} is solved. And by \cite{voisin}, we have:
\begin{cor}
Assume that the Tate conjecture for surfaces over finite fields is true. Then every integral Hodge $(n-1,n-1)$-class on a smooth projective rationally connected $n$-fold is a $\mathbb{Z}$-linear combination of the cohomology classes of rational curves.
\end{cor}

For the proof of the main theorem, one can first reduce rational equivalence to algebraic equivalence (Lemma \ref{alg}). Then in characteristic $0$ one can apply the main construction of \cite{ghs} to the trivial product $X\times \mathbb{P}^1$ and degenerate any curve to a sum of rational curves (this part is already stated in the unpublished note \cite{zong}). Professor J\'anos Koll\'ar pointed out that a more direct and simpler proof is possible by looking at the natural forgetful map
\[
\Phi: \overline{M}_{g,0}(X, [C]) \to \overline{M}_{g,0}
\]
from the Kontsevich moduli space of stable maps to the moduli space of stable curves. Once we know this map is surjective, a degeneration argument proves the main theorem in all characteristics.

We still include a section on the proof by the product trick which started the whole projects with a few remarks and applications.

Next we turn to the study of one cycles on low degree complete intersections. First we have:

\begin{thm}
Let $X$ be a (possibly singular) complete intersection of type $(d_1, \ldots, d_c)$ with  $d_1+\ldots+d_c\leq n-1 $. Then every rational curve on $X$ is algebraically equivalent to an effective sum of lines.
\end{thm}

Then by the same technique proving our main theorem, we have the following.

\begin{thm}
Let $X$ be a smooth complete intersection of type $(d_1, \ldots, d_c)$ over an algebraically closed field $k$. Assume $d_1+\ldots+d_c\leq n-1 $. If $Char(k)=p>0$, also assume that $X$ is separably rationally connected. Then the Chow group of $1$-cycles $CH_1(X)$ is generated by lines.
\end{thm}

\begin{rem}
A general Fano hypersurface in characteristic $p>0$ is separably rationally connected \cite{Yi}. And it is plausible that the same is true for a general Fano complete intersection.
\end{rem}

\begin{cor}
Let $X$ be a smooth separably rationally connected complete intersection of degree $(d_1, \ldots, d_c)$ such that
\[
\sum_{i=1}^c \frac{d_i (d_i+1)}{2}\leq n.
\]
Then the Chow group of $1$-cycles $CH_1(X)$ is isomorphic to $\ZZ$.
\end{cor}
\begin{proof}
Under the assumption, the Chow group of $1$-cycles is generated by lines and any two lines are rationally equivalent since the Fano scheme of lines in $X$ is rationally chain connected (c.f. Proof of Theorem 4.2, Chap. V, \cite{kollarbook}).
\end{proof}

\begin{rem}
One can prove that the Chow group of one cycles is isomorphic to $\ZZ$ for all complete intersections with a slightly worse degree bound (Proposition \ref{prop:RSC}). The result with $\QQ$ coefficients is well-known (Theorem 4.2, Chap. V, \cite{kollarbook}) even without the smoothness assumption.
\end{rem}

\section{Preliminaries}

\subsection{Rationally connected varieties and smoothing of curves}
Recall the following definition.
\begin{defn}
Let $X$ be a proper variety over a field $k$. It is separably rationally connected (SRC) if it is smooth and there is a generic smooth family of $1$--cycles with geometrically rational components:
\[
\begin{array}{ccccc}
 U  & \stackrel{u}{\longrightarrow} &  X\\
 g\downarrow &&  \\
  B  &  &
\end{array}
\eqno{},
\] over $k$ such that the double evaluation map:
\[
U\times_{B} U \stackrel{(u, u)}{\longrightarrow} X\times X
\]
is generically \'etale and dominates $X\times X$. If we drop the \'etale condition, it is called rationally connected (RC). And if we don't require the generic smoothness of $U\to B$, it is called rationally chain connected (RCC).
\end{defn}

\begin{defn}
Let $X$ be a smooth proper variety. A morphism $f: \mathbb{P}^1 \to X$ is free (resp. very free) if $f^{*}T_X$ is non-negative (resp. ample).
\end{defn}

A variety over an algebraically closed field is separably rationally connected if and only if there is a very free rational curve.

\begin{defn}
\label{defcomb}
Let $k$ be an arbitrary field.
A {\it comb} with {\it $n$ teeth} over $k$ is a projective curve
with $n+1$ irreducible components $C_0,C_1,\dots,C_n$ over $\bar k$
satisfying the following conditions:
\begin{enumerate}
\item The curve $C_0$ is defined over $k$.
\item The union $C_1\cup\dots\cup C_n$ is defined over $k$
(Each individual curve may not be defined over $k$).
\item The curves $C_1,\dots,C_n$ are smooth rational curves
disjoint from each other, and each of them meets $C_0$
transversely in a single smooth point of $C_0$
(which may not be defined over $k$).
\end{enumerate}
The curve $C_0$ is called the {\it handle} of the comb,
and $C_1,\dots,C_n$ the {\it teeth}.
A {\it rational comb} is a comb whose handle is a smooth
rational curve.
\end{defn}

One of the most important techniques in studying separably rationally connected varieties is the smoothing of a comb. Given a morphism $f: C\to X$ from a comb $C$ with handle $C_0$, a smoothing of $f$ is a family $\Sigma \to T$ over a pointed curve $(T, 0)$ together with a morphism $F: \Sigma \to X$ such that $\Sigma_0 \cong C, F\vert_{\Sigma_0}=f$ and $\Sigma_t$ is a smooth curve, isomorphic to $C_0$, for a general $t$.

The following result is implicit in the book \cite{kollarbook} and follows from the work of Koll\'ar-Miyaoka-Mori.
\begin{prop}\label{van}
Given a morphism from a smooth projective curve $f_0: C_0 \to X$ to a smooth proper and separably rationally connected variety $X$ over a field $k$, and an integer $d$, there are $p\gg0$ very free rational curves $f_i: C_i \to X, 1 \leq i \leq p$, such that
\begin{enumerate}
 \item $C=C_0\cup C_1\cup...\cup C_p$ is a comb in the sense of \ref{defcomb}. Furthermore, there is a morphism $f: C \to X$ (defined over $k$) and a smoothing of the comb $\Sigma \to T, F: \Sigma \to X$.
\item $H^1(\Sigma_t, F_t^*{T_{X}}\otimes M)=0$ for a general member $F_t: \Sigma_t \to X$ of the smoothing and any line bundle $M$ of degree $d$.
\end{enumerate}
\end{prop}

For the proof, we need the following lemma.
\begin{lem}\label{vanishing}
Let $D$ be a smooth projective connected curve of genus $g$ and $f: D \to X$ be a morphism to a smooth variety. Then given an integer $d$, there is a number $n$ such that for any line bundle $L$ of degree at least $n$, $H^1(D, f^*T_X\otimes L \otimes M)=0$ for any line bundle $M$ of degree $d$.
\end{lem}

\begin{proof}
There is a line bundle $L_0$ on $D$ such that $H^1(D, f^*T_X\otimes L_0)=0$. Choose a line bundle $L_1$ of degree at least $g-d$. Then $H^0(D, \OO_D(L_1+M))$ is non-zero for any line bundle $M$ of degree $d$ since any divisor of degree at least $g$ is rationally equivalent to an effective divisor. Take $E_M$ to be an effective divisor in the linear system $\vert L_1+M \vert $. Then we have the short exact sequence of sheaves
\[
0 \to f^*T_X\otimes \OO(L_0) \to f^*T_X \otimes \OO(L_0+E_M) \to f^*T_X \otimes \OO(L_0+L_1+M) \vert_{E_M} \to 0.
\]
Thus $H^1(D, f^*T_X \otimes \OO(L_0+L_1+M))=0$ for any line bundle $M$ of degree $d$. \

Take $n$ to be $\deg L_0+\deg L_1+g$. Then for any line bundle $L$ whose degree is at least $n$,
$H^0(D, \OO_D(L-L_0-L_1))$ is non-zero, again since any divisor of degree at least $g$ is rationally equivalent to an effective divisor. Take a divisor $E$ in the linear system $\vert L-L_0-L_1 \vert$. We have the following short exact sequence:
\[
0 \to f^*T_X\otimes \OO(L_0+L_1+M) \to f^*T_X \otimes \OO(E+L_0+L_1+M) \to f^*T_X \otimes \OO(L+M) \vert_E \to 0.
\]
Thus $H^1(D, f^*T_X  \otimes \OO(L+M))=0$.
\end{proof}

\begin{proof}[Proof of Proposition \ref{van}]
We can assemble a comb whose handle is $C_0$ and whose teeth are very free curves $C_i, 1\leq i \leq m$. By Theorem 7.9, Chap II, \cite{kollarbook}, at least $m-h^1(C_0, f^*T_X)$ teeth can be smoothed. We choose $m$ to be large enough so that $m-h^1(C_0, T_X\vert_C)\geq n$, where $n$ is the number in Lemma \ref{vanishing}. Then a general smoothing $\Sigma_t$ of the subcomb has $H^1(\Sigma_t, F_t^* T_X \otimes M)=0$ for any line bundle $M$ of degree $d$ by Lemma \ref{lem:van} (Take $E$ to be $F^*T_X \otimes N$, where $N$ is the line bundle on $\Sigma$ whose restriction to a general fiber and the handle $C_0$ is isomorphic to $M$ and the restriction to the teeth are trivial line bundles).
\end{proof}
\begin{lem}[\cite{kollarbook}, Chap. II, Lemma 7.10.1]\label{lem:van}
Let $C=C_0 \cup C_1 \cup C_2...\cup C_p$ be a comb with $p$ teeth as in Definition \ref{defcomb}. Let $q: \Sigma \to T$ be a smoothing of $C$. And let $E$ be a vector bundle over $\Sigma$ such that the restrictions $E|_{C_i}$ are all ample for $1 \leq i \leq p$ and $H^1(C_0, L \otimes E|_{C_0})=0$ for every line bundle $L$ on $C_0$ of degree at least $p$. Then $H^1(\Sigma_{t}, E|_{\Sigma_{t}})=0$ for a general $t$.
\end{lem}

\begin{rem}\label{nodal}
Let $f: D \to X$ be a curve in a separable rationally connected variety $X$. Assume $D$ is embedded. One can add very free rational curves to form a higher genus  nodal curve and smooth it, as is done in Section 2.2, \cite{ghs}. When $D$ is not embedded in $X$, one can first embed $D$ in $X \times \mathbb{P}^3$ and then do the above operation and project to $X$. In other word, after adding suitable rational curves, we can deform the reducible curve to a curve of higher genus in $X$.
\end{rem}

\subsection{Moduli space of stable maps to projective varieties}
\begin{defn}\label{defstable}
Let $C$ be a connected nodal curve, and $X$ be a projective variety. We call a morphism $f: C \to X$ a stable map if every irreducible component of $C$ which is mapped to a point is one of the followings:
\begin{itemize}
\item A curve of arithmetic genus at least $2$.
\item A curve of arithmetic genus $1$ having at least $1$ intersection point with other components of $C$.
\item A curve of arithmetic genus $0$ having at least $3$ intersection points with other components of $C$.
\end{itemize}

\end{defn}
Let $\beta$ be a curve class of $X$. We have a proper moduli space of stable maps $\overline{M}_{g, 0}(X, \beta)$ (\cite{FP}). When $X$ is a point, we recover the moduli space of stable curves  $\overline{M}_{g,0}$ as in \cite{deligne-mumford}.

\begin{rem}
When $X$ is not projective, we might not have a proper moduli stack of stable maps $\overline{M}_{g, 0}(X, \beta)$: given a family of stable maps over the generic point of a discrete valuation domain to a proper, non-projective variety, it is always possible to extend the family over the closed point (possibly after a base change) to a family of prestable maps (i.e. maps from a family of nodal curves). But it is not clear that one can extend the family over the closed point (even after a base change) to a family of stable maps.
\end{rem}

Let $X$ and $Y$ be projective varieties with a morphism $\pi:X \to Y$. Fix a curve class $\beta$ of $X$. We have then a natural
forgetful map (\cite{BM}):
$$
\Phi : \overline{M}_{g,0}(X,\beta) \to \overline{M}_{g,0}(Y,\pi_{*}\beta)
$$
defined by composing a map $f : C \to X$ with $\pi$ and collapsing
components of $C$ as necessary to make the composition $\pi \circ f$ stable.

Now we have an easy but important observation.
\begin{lem}\label{contraction}
For a stable map $f : C \to X$, the components that are contracted under $$
\Phi: \overline{M}_{g,0}(X,\beta) \to \overline{M}_{g,0}(Y,\pi_{*}\beta)
$$ are all rational.
\end{lem}
\begin{proof}
By Definition \ref{defstable}, the only possible non-stable components that need to be contracted are smooth rational curve having at most $2$ intersection points with other components of $C$.
\end{proof}

\section{Reducing rational equivalence to algebraic equivalence}
The following proposition is essentially Proposition 3.13.3, Chap. IV, \cite{kollarbook}. The existence of the number $N$ follows from the argument there but is not explicitly stated. We include the proof here for completeness.
\begin{prop}\label{kollarlemma}
Let $X$ be a proper rationally chain connected variety over an algebraically closed field. Then there is a positive integer $N$ together with a family of effective $1$-cycles with rational components
\[
\begin{array}{ccccc}
 F & \stackrel{u}{\longrightarrow} &  X\\
 g\downarrow &&  \\
  B  &  &
\end{array}
\eqno{},
\]
where $u$ is the evaluation map, such that for any $1$-cycle $D$ in $X$, there are integers $m_i$'s and rational curves $F_i$'s in the fibers of $g: F \to B$, which satisfies the following rational equivalence relation: $$N \cdot D \sim_{r.e.}\sum_i m_i u_*(F_i).$$
\end{prop}

\begin{proof}
 By the definition of rational chain connectedness, there is a family of connected effective $1$-cycles with rational components $g: F \to B$ in $X$ with the evaluation map $u: F \to X $, and such that the map $$m: F\times_{B} F \to X\times X$$ is generically finite of degree $N$ and surjective. 

Now pick a general point $x_{0}$ of $X$. Then $u_0: u^{-1}(x_0) \times_{B} F \to {x_0}\times X$ is generically finite of degree $N$ and surjective. Denote $ u^{-1}(x_0) \times_{B} F $ by $F_0$. And let $B_0=g(u^{-1}(x_0))$. Then we have a diagram:
$$
\begin{array}{ccccc}
 F_0  & \stackrel{u_0}{\longrightarrow} &  X\\
 g\downarrow &&  \\
  B_0  &  &
\end{array}
\eqno{}
$$
Fix an irreducible curve $C$ in $X$. Then the class $[C]$ is rationally equivalent to a class of the form $[C_1]-[C_2]$, where $C_1$ and $C_2$ are curves in $X$ containing a general point. To see this, first find a projective birational morphism $X' \to X$ which is an isomorphism near the generic point of $C$. Then there is a unique lifting of $C$ to $X'$. Thus it suffices to prove this when $X$ is projective. Then one can find an irreducible curve $C'$, which contains a general point, as a residual curve of $C$ such that $C\cup C'$ is a complete intersection of very ample divisors. Then $[C]$ is rationally equivalent to the difference of a general complete intersection and $C'$. So we may assume that $C$ contains a general point of $X$. Let $\hat{C}$ be the one dimensional component of the inverse image of $C$ under ${u_0}$. Then  ${u_{0}}_{*}[\hat{C}]\sim_{r.e.} N\cdot C$.

The curve $\hat{C}$ is a section of the chain  of ruled surfaces, intersecting each other at a section:
$$ F_0 \times_{B_0} \hat{C}  \to \hat{C}.$$
There is another section $C_0:u^{-1}(x_0) \times_{B_0} \hat{C}$, which is mapped to $x_0$ via the evaluation map $u_0$.

Since two sections of a ruled surface are rationally equivalent modulo rational curves in the fiber, we have $$ \hat{C} - C_0 \sim_{r.e.} \sum_i n_i F_i$$  where $F_i$'s are irreducible components of the fibers of $F\to B$. Push forward this relation to $X$. Since $u_0(C_0)={x_0}$, we get $$N\cdot C \sim_{r.e.} \sum_i m_i F_i.$$
\end{proof}

Let $CH_1(X)_{alg}$ be the subgroup of the Chow group of $1$-cycles on $X$ which are algebraically equivalent to $0$. Here is a classical lemma.
\begin{lem}[\cite{bloch-ogus}, Lemma 7.10 ]\label{divisible}
The group $CH_p(X)_{alg}$ is a divisible group. Namely, for any integer $N>0$ and any element $x$ in $CH_p(X)_{alg}$, there is another element $y$ in $CH_p(X)_{alg}$ such that $x=N \cdot y$.
\end{lem}
\begin{proof}
The subgroup $CH_p(X)_{alg}$ is generated by cycles of the form $\Sigma_p-\Sigma_q$, where $\Sigma \subset C \times X$ is a family of cycles over a smooth projective curve $C$. So the divisibility of $CH_p(X)_{alg}$ follows from the divisibility of $CH_0(C)_{alg}$, which is isomorphic to the Jacobian of the curve $C$.
\end{proof}

Denote by $\cdot N$ the map of multiplication by $N$ in $CH_1(X)_{alg}$. Then $$CH_1(X)_{alg} \stackrel{\cdot N}{\longrightarrow} CH_1(X)_{alg}$$ is surjective by the divisibility of $CH_1(X)_{alg}$. Together with Proposition \ref{kollarlemma}, one immediately has:
\begin{cor}\label{sur}
Notation as in Proposition \ref{kollarlemma}. Then $CH_1(X)_{alg}$ is generated by cycles with only rational components.
\end{cor}
This allows us to reduce Theorem \ref{thm:main} to a similar statement for cycles modulo algebraic equivalence.
\begin{lem}\label{alg}
In order to prove the main theorem \ref{thm:main}, it suffices to prove that rational curves generate the group of one cycles modulo algebraic equivalence.
\end{lem}
\begin{proof}
Assume that for any curve $C$ in $X$, there are rational curves $C_i$'s and integers $n_i's$, such that $$[C]-\sum_i n_i [C_i]$$ is algebraically equivalent to $0$, namely, an element in $CH_1(X)_{alg}$. Then by Corollary \ref{sur}, we have rational equivalence relation $$[C]-\sum_i n_i [C_i]\sim_{r.e.} \sum_j m_j [F_j]$$ where $m_i$'s are integers and $F_j$'s are rational curves. So $C$ is rationally equivalent to an integral sum of rational curves.
\end{proof}

\section{Proof of the main theorem over $\mathbb{C}$ with a product trick}

We can use a product trick to prove the main result of this paper over the field of complex numbers $\mathbb{C}$, by applying the main argument of the celebrated paper ``Families of Rationally Connected Varieties'' of T. Graber, J. Harris and J. Starr (\cite{ghs}).

\begin{theorem}\label{maintheorem1}  Let $X$ be a smooth proper rationally connected variety over $\mathbb{C}$. Then every curve on $X$ is algebraically equivalent to a $\mathbb{Z}$-linear combination of rational curves.
\end{theorem}

\begin{rem}\label{red}
It is easy to reduce to the projective case as follows: by Chow's lemma and resolution of singularities, there is a smooth projective variety $X'$, with proper birational morphism $f: X' \to X$. One can smooth any curve $C\subset X$ (after adding free rational curves) to a general curve $C'$ not supported in the exceptional locus of $f$. Then just take a lift of $C'$ in $X'$.
\end{rem}

Assuming $X$ is projective, the idea of proof is to first lift any irreducible curve $C$ in $X$ to $X \times \mathbb{P}^1$. By \cite{ghs}, there are very free rational curves which are horizontal with respect to the projection to $\mathbb{P}^1$, just add enough of these curves to form a comb and then smooth to a curve $\hat{C}$ which is $flexible$ in the sense of \cite{ghs}, i.e. the natural map $$\overline{M}_{g,0}( X\times \mathbb{P}^1, \beta) \to \overline{M}_{g,0}(\mathbb{P}^1, d)$$  is proper and surjective at the corresponding component. Degenerating $\hat{C}$ to a sum of rational curves and pushing forward this relation to $X$, we get the desired result.

The following subsection is largely a restatement of the unpublished note \cite{zong}.

\subsection{Families of rationally connected varieties}

Let $Y$ be a smooth projective variety with a morphism $Y \to \mathbb{P}^1$ whose general fibers are rationally connected.

For a class $\beta \in H_2(Y,\mathbb{Z})$ having intersection
number $d$ with a fiber of the map $\pi$.  We have then a natural
morphism:
$$
\varphi : \overline{M}_{g,0}(Y,\beta) \to \overline{M}_{g,0}(\mathbb{P}^1,d).
$$
\begin{defn}
Let $f : C
\to Y$ be a stable map from a nodal curve $C$ of genus $g$ to $X$ with
class $f_*[C] = \beta$. We say that $f$ is {\it flexible} relative to
$\pi$ if the map $\varphi : \overline{M}_{g,0}(Y,\beta) \to \overline{M}_{g,0}(\mathbb{P}^1,d)$ is dominant at
the point $[f] \in \overline{M}_{g,0}(Y,\beta)$ and $\pi: C \to \mathbb{P}^1$ is flat.
\end{defn}

\begin{prop}\label{degeneration}
A  {\it flexible} curve $f : C \to Y$ can be degenerated to an effective sum of rational curves in $Y$.
\end{prop}
\begin{proof}
It is a classical fact that the variety $\overline{M}_{g,0}(\mathbb{P}^1,d)$ has a unique
irreducible component whose general member corresponds to a flat map
$f : C \to \mathbb{P}^1$, see \cite{F}. Since the map
$\varphi : \overline{M}_{g,0}(Y,\beta) \to \overline{M}_{g,0}(\mathbb{P}^1,d)$ is proper and dominant, $\varphi$ is
surjective on the component of $\pi : C \to Y$.  By Lemma \ref{contraction} it is enough to find a degeneration of $C \to \mathbb{P}^1$ in $\overline{M}_{g,0}(\mathbb{P}^1,d)$ as a sum of rational curves, which is elementary.
\end{proof}

\begin{theorem}\label{main}
(Main Construction of \cite{ghs})For any multisection $B \to \mathbb{P}^1$, there are rational curves $C_i$'s such that $B\cup C_1 \cup...\cup C_m$ can be deformed to a flexible curve of $Y \to \mathbb{P}^1$.
\end{theorem}

Now we can prove Theorem \ref{maintheorem1}.

\begin{proof}[Proof of Theorem \ref{maintheorem1}]
Take $Y=X\times \mathbb{P}^1$, for any irreducible curve $C\subset X$, lift it to a curve $C'$ in $X\times{0} \subset Y$. Since $Y$ is rationally connected, we can add enough free curves of $Y$ which are horizontal with respect to the projection $Y\to \mathbb{P}^1$, such that the comb can be deformed to a multisection $\Gamma$ of the fibration $Y\to \mathbb{P}^1$. By Theorem \ref{main}, we can add some other rational curves to $\Gamma$ and deform the reducible curve to a flexible curve. Then by Proposition \ref{degeneration}, it can be degenerated to a sum of rational curves, so $C'$ is algebraically equivalent to an integral sum of rational curves in $Y$. Pushing forward the relation to $X$ finishes the proof.
\end{proof}

\begin{rem}
It might be possible to prove the main theorem in all characteristics by applying the argument of \cite{starr-jong} to $X\times \mathbb{P}^1$.
\end{rem}

\section{Proof of the main theorem in all characteristics }
We begin with a lemma.
\begin{lem}\label{surjectivity}
Let $f: C \to X$ be a morphism from a smooth projective connected curve $C$ of genus $g (\geq 2)$ to a projective variety $X$. Assume that the image of $f$ lies in the smooth locus $X^\text{sm}$ of $X$ and $H^1(C, f^*T_X)=0$. Then the moduli space $\overline{M}_{g, 0}(X, [C])$ is smooth at the point represented by $(f: C \to X)$ and the natural forgetful map:
\[
\Phi: \overline{M}_{g, 0}(X, [C]) \to \overline{M}_{g,0},
\]
restricted to the (unique) irreducible component containing the point represented by $(f: C \to X)$, is surjective.
\end{lem}

\begin{proof}
The deformation and obstruction space of the stable map $(f: C \to X)$ is
\[
Def(f)=\mathbb{H}^1(C,\mathbb{R}Hom_{\OO_C}(\Omega_f^{\cdot},\OO_C))
\]
 \[
Obs(f)=\mathbb{H}^2(C,\mathbb{R}Hom_{\OO_C}(\Omega_f^{\cdot},\OO_C))
\]
where
$\Omega^{\cdot}_f$ is the complex
\[
\begin{CD}
-1 & & 0 \\
f^*\Omega_X @> df^\dagger >> \Omega_C.
\end{CD}
\]
We have the long exact sequence
\begin{align*}
0 &\to H^0(C,T_{C}) \to H^0(C, f^*T_{X}) \to  \bH^1({{\RR}Hom_{\OO_C}(\Omega^{\cdot}_f,\OO_C)})\\
 &\to H^1(C,T_C) \to H^1(C, f^*T_{X}) \to \bH^2({{\RR}Hom_{\OO_C}(\Omega^{\cdot}_f, \OO_C)})\to 0
\end{align*}
Note that $\bH^1({{\RR}Hom_{\OO_C}(\Omega_f^{.},\OO_C)}) \to H^1(C, T_C)$ is the map between tangent spaces of $\overline{M}_{g, 0}(X, [C])$ and $\overline{M}_{g,0}$. Thus the vanishing of $H^1(C, T_X |_{C})$ implies both the surjectivity of the tangent space map and smoothness of $\overline{M}_{g,0}(X, [C])$ at the point represented by $(f: C \to X)$. Therefore the restriction of the forgetful morphism $\Phi$ to the component containing the point $(f: C \to X)$ is dominant. Since the morphism $\Phi$ is also proper when restricted to this component, it is also surjective.
\end{proof}

Now we can finish the proof of the main theorem by the following result and Lemma \ref{alg}.
\begin{theorem} Let $X$ be a smooth proper and separably rationally connected variety over an algebraically closed field of arbitrary characteristic.
Then every curve on $X$ is algebraically equivalent to a $\mathbb{Z}$-linear combination of rational curves.
\end{theorem}
\begin{proof}
Choose an irreducible curve $f: C \to X$. We may assume the genus of $C$ is at least $2$ by Remark \ref{nodal}. Then by Proposition \ref{van}, we may also assume that $H^1(C, f^* T_X(-p))=0$ for any point $p$ in $C$ (up to adding very free rational curves and smoothing). Note that $H^1(C, f^* T_X)$ also vanishes in this case. By Chow's lemma, there is a normal projective variety $X'$ with a birational morphism $\pi: X' \to X$. The exceptional locus of $\pi$ has codimension at least $2$ in $X$. Thus there is a open subset $U$ of $X$, whose compliment has codimension at least $2$, and is isomorphic to an open subset $V$ of $X'$. Since $H^1(C, f^* T_X(-p))=0$, a general deformation of $C$ lies in $U$. By abuse of notations, we still write the general deformation as $C$. Then by upper semi-continuity, $H^1(C, f^* T_U)=H^1(C, f^* T_X)=H^1(C, f^*T_{X'})=0$. And it suffices to prove the statement for the map $f:C \to U \cong V \subset X'$.

Finally Lemma \ref{surjectivity} implies that the forgetful map
\[
\Phi: \overline{M}_{g,0}(X', [C]) \to \overline{M}_{g,0}
\]
 is surjective when restricted to the irreducible component of the Kontsevich moduli space of stable map $\overline{M}_{g, 0}(X', [C])$ containing the point represented by $(f: C \to X')$. The moduli space $\overline{M}_{g,0}$ is irreducible \cite{deligne-mumford}. So one can specialize the image of $C$ in $\overline{M}_{g,0}$ to a singular stable curve whose irreducible components are all rational curves. And then choose any preimage of the point in the irreducible component of $\overline{M}_{g,0}(X', [C])$ containing $(f: C \to X')$. By Lemma \ref{contraction}, the preimage is represented by a map from a curve with only rational components to $X'$. Thus the curve class $[C]$ is algebraically equivalent to the class of a union of rational curves.
\end{proof}
\begin{rem}
An elliptic curve without any marked point is not stable. So we increase the genus first to get a stable curve. One can also use $\overline{M}_{g,1}$ and slightly modify Lemma \ref{surjectivity} to adapt to this case. 
\end{rem}

\section{One cycles on low degree complete intersections}

In this section, we prove the following theorem by the same technique as in previous sections.
\begin{thm}
Let $X$ be a smooth complete intersection of type $(d_1, \ldots, d_c)$ over an algebraically closed field $k$. Assume $d_1+\ldots+d_c\leq n-1 $. If $Char(k)=p$, also assume  that $X$ is separably rationally connected. Then the Chow group of $1$-cycles $CH_1(X)$ is generated by lines.
\end{thm}

By the assumption, $X$ is rationally chain connected by chains of lines (Lemma 4.8.1, Chap. V of \cite{kollarbook}). Let $F(X)$ be the Fano scheme of lines of $X$.  Then
\begin{enumerate}
\item $CH_0(F(X))\otimes \QQ \to CH_1(X) \otimes \QQ$ is surjective ( Proposition 3.13.3, Chap. IV \cite{kollarbook}).
\item The cokernel of $CH_0(F(X))_{alg} \to CH_1(X)_{alg}$ is annihilated by an integer $N$ (c.f. Proposition \ref{kollarlemma}).
\end{enumerate}
Thus it suffices to show that every curve is algebraically equivalent to an integral sum of lines by the same argument as in Lemma \ref{alg}. By the main theorem \ref{thm:main}, it suffices to show that every rational curve is algebraically equivalent to an integral sum of lines.

\begin{thm}\label{thm:hypersurface}
Let $X$ be a (possibly singular) complete intersection of type $(d_1, \ldots, d_c)$ with  $d_1+\ldots+d_c\leq n-1 $ in $\PP(V) \cong \PP^n$. Then every rational curve on $X$ is algebraically equivalent to an effective sum of lines.
\end{thm}
 We will need the following connectedness result of Hartshorne.
\begin{lem}[(Hartshorne, \cite{Connectedness})]\label{connectedness}
Let $X$ be a subscheme in $\PP^N$ defined by $M$ homogeneous polynomials. And let $Y$ be a closed subset of $X$ with dimension less than $N-M-1$. Then $X-Y$ is connected.
\end{lem}

\begin{proof}[Proof of Theorem \ref{thm:hypersurface}]
We use induction on the degree of the considered rational curve. The degree $1$ case is obvious.

Assume the statement is true for all rational curves whose degree is less than $e (\geq 2)$. Fix a copy $C$ of $\PP^1$, and consider the parameter space
\[
\PP:= \PP( Hom(V^*,H^0(C, \OO(e))))
\]

If we choose homogeneous coordinates $x_0,\ldots, x_n$ on $\PP(V)$, then $\PP$ parametrizes $(n+1)$-tuples $[u_0,...,u_n]$ of homogeneous degree $e$ polynomials on $C$.

Let $X$ be a complete intersection of codimension $c$ of type $(d_1, \ldots, d_c)$ defined by nonzero degree $d_i$ homogeneous polynomials $F_i (1 \leq i \leq c)$ on $\PP(V)$.  Define $\PP_X$ to be the closed subset of $\PP$ parameterizing $[u_0,...,u_n]$ such that $F_i(u_0,...,u_n)$ equals $0$ for all $1 \leq i \leq c$.  The subvariety $\PP_X$ is defined by homogeneous polynomials of degree $$\underbrace{d_1, \ldots, d_1}_{ed_1+1}, \ldots, \underbrace{d_c, \ldots, d_c}_{ed_c+1} $$ on $\PP$.

Let $B \subset \PP$ be the closed subvariety parameterizing tuples $[u_0, \ldots, u_n]$ where $span\{u_0,...,u_n\}$ in $H^0(C, \OO(e))$ is $1$-dimensional, i.e., every pair $u_i, u_j$ satisfies a scalar linear relation.  The codimension of $B$ in $\PP$ is $ne$.

Let $D \subset \PP$ be the closed subvariety parameterizing $(n+1)$-tuples $[u_0,...,u_n]$ that have a common zero in $C$. Clearly $D$ contains $B$. The point outside $D$ parametrizes a degree $e$ map from $C$ to $\PP^n$.

Now we see that $\PP_X \cap B$ parameterizes a degree $e$ polynomial in two variables (up to scaling) and a point in $X$. So its dimension is $n-c+e$. Then in the situation of Lemma \ref{connectedness}, take $Y$ to be $\PP_X \cap B$ and
$N=ne+n+e$, $M=e(d_1+...+d_c)+c$. The condition $n-c+e <N-M-1$ is simply $$e(n-d_1-d_2-...-d_c)>1.$$ Therefore $\PP_X - B$ is connected for all $e\geq 2$.

Now let $u = [u_0, \ldots, u_n]$ be any element of $\PP_X - D$, e.g., a parameterized, degree $e$ morphism from $C$ to $X$. Let $u'$ be a degree $e$ multiple cover of a line in $X$. Then $u'$ is algebraically equivalent to a union of lines.  Since $\PP_X-B$ is connected, there exists a connected curve mapping to $\PP_X-B$ whose image
connects these two points.  We may assume that it is a chain of irreducible components, with each consecutive pair meeting in a single node.

If none of the nodes is contained in $D$, then after deleting finitely many points from the curve, none of which are the nodes, we may assume the entire connected curve is disjoint from $D$.  Thus this connected curve parameterizes a family of stable maps.  Then $u$ is algebraically equivalent to $u'$, and hence algebraically equivalent to a union of lines.

So it remains to consider the case when some of the nodes lie in $D$ but not in $B$. Consider the first such node, starting from the component containing $u$. Approaching the node from the component closer to u, we have a family of elements in the complement of $D$ which approaches a point of $D$, i.e., we have a family of stable maps which approaches a point of $D$.  But this is a point of $D$ which is not in $B$.  So the corresponding rational map from $C$ to $X$ is non-constant of degree strictly less than $e$.  It follows that the limiting stable map has at least one component of positive degree which is strictly less than $e$, i.e., this is a point in the boundary.  But then, each component of this boundary stable map has degree strictly less than $e$, hence by induction each component is algebraically equivalent to a union of lines. Thus $u$ is also algebraically equivalent to a union of lines.
\end{proof}

\begin{rem}[(suggested by C. Voisin)]\label{rem:v}
The argument presented here proves that for a smooth Fano complete intersection of type $(d_1, \ldots, d_c)$ such that $d_1+\ldots+d_c\leq n-1$, the Griffiths group of one cycles homologous to $0$ modulo algebraic equivalence is trivial. Indeed, this follows from the statement that every one cycle is rationally equivalent to a linear combination of lines and the fact that the Fano scheme of lines is connected in all but one case, i.e. the quadric surface in $\PP^3$ (\cite{kollarbook}, Theorem 4.3, Chap. V, \cite{zongb}, Theorem 1.3). But the statement is easily seen to be true for the quadric surface.
\end{rem}

\section{Remarks on the product trick}
The product trick gives a weak form of Proposition \ref{kollarlemma}.

\begin{prop}\label{product1}
Let $X$ be a proper rationally connected variety over an algebraically closed field, then $CH_1(X)\otimes_{\mathbb{Z}} \mathbb{Q}$ is generated by rational curves.
\end{prop}

\begin{proof}
For any curve $C$ in $X$, consider the graph $\Gamma \subset C \times X$. Take any point $x\in X$. Let $C'$ be the trivial section $x\times C\subset X\times C$. Since $X\times_C K(C)$ is a rationally connected variety over the function field $K(C)$ of $C$, after a base change $S\to C$ there is a rational curve in $X\times_{C} K(S)$ connecting $\Gamma \times_{C} K(S)$ and $C' \times_{C} K(S')$. Thus there is a ruled surface with two sections $\Gamma$ and $C'$. Recall again the fact that two sections of a ruled surface are rationally equivalent modulo fibers, which are all rational curves. The result follows from pushing forward everything to $X$ and the fact that the push-forward of $C'$ is $0$ in the Chow group of $X$.
\end{proof}

We give one more application of the product trick. Although, strictly speaking, the trick is not necessary, it does make the proof more transparent.

\begin{prop}\label{prop:RSC}
Let $X \subset \PP^n$ be a closed subscheme defined by equations of degree $d_1, \ldots, d_c$ such that $\sum_{i=1}^c d_i^2 \leq n$. Then $CH_1(X) \cong \ZZ$.
\end{prop}

\begin{proof}
Let $f: C \to X$ be a curve in $X$. We first show that $C$ is rationally equivalent to a sum of lines in $X$.

Consider the graph $\Gamma$ of the map in $C \times X$. Choose a trivial section $S=C \times x$ of $C \times X \to C$, where $x$ is a closed point in $X$. Now think of the fibration over $C$ as a scheme $\mcX$ defined over $K(C)$, the function field of $C$, by equations of degree $d_1, \ldots, d_c$ and the graph and the trivial section as two rational points of $\mcX$. Under the assumptions on $d_i$'s, $X$ is rationally chain connected by chains of two lines. And furthermore, the scheme parameterizing the chain of two lines containing two points is defined by equations of degree
\[
(1, 2, \ldots, d_1-1, 1, 2, \ldots, d_1-1, d_1, \ldots, 1, 2, \ldots, d_c-1, 1, 2, \ldots, d_c-1, d_c).
\]
To see this, just consider the equations for the intersection point of the two lines. In particular, the scheme parameterizing chains of two lines containing the two rational points is a scheme defined over $K(C)$ by equations of the above degrees. Then by Tsen's theorem, there is a $K(C)$-rational point in the scheme. The chain of lines between the two rational points is equivalent to a chain of ruled surfaces over $C$ containing the two sections $\Gamma$ and $S$ as sections (of the ruled surface). So
\[
\Gamma \sim_{r.e.} S+ \text{lines}
\]
in $C \times X$. Then push forward the relations to $X$. Notice that the push-forward of $[S]$ is $0$. Thus we get the desired relation in $X$.

Finally note that any two lines in $X$ are rationally equivalent since the Fano scheme of $X$ is rationally chain connected under the assumptions on $d_i$.
\end{proof}

\begin{acknowledgements}
 We thank Professor Burt Totaro for introducing the question to us by his enlightening lectures, correcting many mistakes in the first draft of this paper, and helping us to form the final argument, Professor Claire Voisin for the argument of reducing rational equivalence to algebraic equivalence and Remark \ref{rem:v}, Professor J\'anos Koll\'ar for his constant support for the second named author and enlightening comments on the proof, Professor Jason Starr for helpful discussions about rational curves on Fano complete intersections, and Professor Chenyang Xu for very encouraging comments during the project.

\end{acknowledgements}


\begin{thebibliography}{AK03}

\bibitem[AK03]{araujo} C. Araujo and J. Koll\'ar,
    Rational Curves on Varieties, Higher Dimensional Varieties and rational points (Budapest 2001), Bolyai Soc. Math. Stud., vol. 12, Springer, Berlin, 2003, pp. 13-68.
\bibitem[BO74]{bloch-ogus} S.Bloch, O.Ogus, Gersten's conjecture and the homology of schemes, Annales scientifiques de l'\'Ecole Normale Suprieure, 4e s\'erie, tome 7, no 2 (1974), p. 181-201.


\bibitem[BM96]{BM}
K. Behrend and Yu. Manin.
\newblock Stacks of stable maps and Gromov-Witten invariants.
\newblock Duke Math. J. Volume 85, Number 1 (1996), 1-60.

\bibitem[dJS03]{starr-jong} A.J. de Jong, J. Starr, Every rationally connected variety over the function field of a curve has a rational point, American Journal of Mathematics, {\bf 125}, 567-580 (2003).

\bibitem[DM69]{deligne-mumford}P. Deligne and D. Mumford, The irreducibility of the space of curves of given genus, Publications Mathématiques de l'IHÉS (Paris), 1969, 36: 75–109.

\bibitem[Flo]{Floris}
E.~Floris.
\newblock Fundamental divisors on Fano varieties of index n-3.
\newblock {\em to appear in Geometriae Dedicata.}

\bibitem[Ful69]{F}W. Fulton, Hurwitz Schemes and Irreducibility of Moduli of Algebraic Curves, Ann. of Math., Vol. 90, No. 3, Nov., 1969, page 542-575.

\bibitem[FP97]{FP}W. Fulton and R. Pandharipande, Notes on stable maps and quantum
cohomology, Notes on stable maps and quantum cohomology. Algebraic
geometry Santa Cruz 1995, 45-96, Proc. Sympos. Pure Math., 62, Part
2, Amer. Math. Soc., Providence, RI, 1997.

\bibitem[GHS03]{ghs} T. Graber, J. Harris and J. Starr, Families of rationally connected varieties.
J. Amer. Math. Soc., 16(1), 57-67  (2003).


\bibitem[Har62]{Connectedness}
R. Hartshorne, Complete Intersections and Connectedness, American Journal of Mathematics, Vol. 84, No. 3 (Jul., 1962), pp. 497-508




\bibitem[HV11]{HoringVoisinFano}
Andreas H{\"o}ring and Claire Voisin.
\newblock Anticanonical divisors and curve classes on {F}ano manifolds.
\newblock {\em Pure Appl. Math. Q.}, 7(4, Special Issue: In memory of Eckart
  Viehweg):1371-1393, 2011.

 \bibitem[Kol96]{kollarbook} J\'{a}nos Koll\'{a}r.
\newblock {\em Rational curves on algebraic varieties}, volume~32 of {\em
  Ergebnisse der Mathematik und ihrer Grenzgebiete. 3. Folge. A Series of
  Modern Surveys in Mathematics [Results in Mathematics and Related Areas. 3rd
  Series. A Series of Modern Surveys in Mathematics]}.
\newblock Springer-Verlag, Berlin, 1996.

 \bibitem[Kol10]{kollarportugaliae} J. Koll\'ar. Holomorphic and pseudo-holomorphic curves on rationally connected varieties, Portugaliae Mathematica, Volume 67, Issue 2, (2010) 155–179.

     \bibitem[Voi06]{voisinuniruled} C. Voisin,  On integral Hodge classes on uniruled and Calabi-Yau
threefolds, in {\it Moduli Spaces and Arithmetic Geometry},
Advanced Studies in Pure Mathematics 45, pp. 43-73, (2006).

\bibitem[Voi12] {voisin}C. Voisin, Remarks on curve classes on rationally connected varieties. arXiv:1201.1072.
\bibitem[Zhu11]{Yi} Y. Zhu, Fano hypersurfaces in positive characteristic.  arXiv:1111.2964.
\bibitem[Zon12a]{zong} Hong R. Zong, Curve classes on rationally connected varieties. arXiv:1207.0575.
\bibitem[Zon12b]{zongb} Hong R. Zong, On the Space of Conics on Complete Intersections. arXiv:1211.1946.


\end{thebibliography}
\end{document}